\documentclass[12pt]{amsart}
\usepackage{amssymb, hyperref, mathrsfs,xcolor, cancel, enumerate, comment}

\usepackage[margin=1.0in]{geometry}

\newtheorem{theorem}{Theorem}[section]

\newtheorem{proposition}[theorem]{Proposition}
\newtheorem{example}[theorem]{Example}
\newtheorem{assumption}[theorem]{Assumption}

\newtheorem{definition}[theorem]{Definition}
\newtheorem{corollary}[theorem]{Corollary}

\theoremstyle{remark}
\newtheorem{remark}[theorem]{Remark}

\newcommand{\calE}{\ensuremath{\mathcal{E}}}
\newcommand{\calF}{\ensuremath{\mathcal{F}}}
\newcommand{\calG}{\ensuremath{\mathcal{G}}}
\newcommand{\calL}{\ensuremath{\mathcal{L}}}
\newcommand{\calR}{\ensuremath{\mathcal{R}}}
\newcommand{\calO}{\ensuremath{\mathcal{O}}}

\newcommand{\calS}{\ensuremath{\mathcal{S}}}
\newcommand{\calT}{\ensuremath{\mathcal{T}}}
\newcommand{\calH}{\ensuremath{\mathcal{H}}}

\newcommand{\calC}{\ensuremath{\mathcal{C}}}

\newcommand{\gln}{\ensuremath{\operatorname{GL}}}

\newcommand{\mo}{{-1}}
\newcommand{\scrC}{\ensuremath{\mathscr{C}}}
\newcommand{\scrL}{\ensuremath{\mathscr{L}}}

\newcommand{\bbZ}{\ensuremath{\mathbb{Z}}}
\newcommand{\bbC}{\ensuremath{\mathbb{C}}}

\newcommand{\bbR}{\ensuremath{\mathbb{R}}}

\newcommand{\bbN}{\ensuremath{\mathbb{N}}}

\begin{document}

 \title{On non-minimal complements}
 
\author{Arindam Biswas}
\address{Department of Mathematics, Technion - Israel Institute of Technology, Haifa 32000, Israel}
\curraddr{}
\email{biswas@campus.technion.ac.il}
\thanks{}

\author{Jyoti Prakash Saha}
\address{Department of Mathematics, Indian Institute of Science Education and Research Bhopal, Bhopal Bypass Road, Bhauri, Bhopal 462066, Madhya Pradesh,
India}
\curraddr{}
\email{jpsaha@iiserb.ac.in}
\thanks{}



\begin{abstract}
The notion of minimal complements was introduced by Nathanson in 2011. Since then, the existence or the inexistence of minimal complements of sets have been extensively studied. Recently, the study of inverse problems, i.e., which sets can or cannot occur as minimal complements has gained traction. For example, the works of Kwon, Alon--Kravitz--Larson, Burcroff--Luntzlara and also that of the authors, shed light on some of the questions in this direction. These works have focussed mainly on the group of integers, or on abelian groups.  In this work, our motivation is two-fold:
\begin{enumerate}
	\item to show some new results on the inverse problem,
	\item to concentrate on the inverse problem in not necessarily abelian groups.
\end{enumerate} As a by-product, we obtain new results on non-minimal complements in the group of integers and more generally, in any finitely generated abelian group of positive rank and in any free abelian group of positive rank. Moreover, we show the existence of uncountably many subsets in such groups which are ``robust'' non-minimal complements.
\end{abstract}

\maketitle

\section{Introduction}
\subsection{Motivation}
 Given two nonempty subsets $A, B$ of a group $G$, the set $A$ is said to be a left (resp. right) complement to $B$ if $A \cdot B = G$ (resp. $B\cdot A = G$). If $A$ is a left (resp. right) complement to $B$ and no subset of $A$ other than $A$ is a left (resp. right) complement to $B$, then $A$ is said to be a minimal left (resp. right) complement to $B$. The study of minimal complements began with Nathanson in \cite{NathansonAddNT4}, who introduced the notion in the context of additive number theory as a natural arithmetic analogue of the metric concept of nets. Since then, most of the literature about minimal complements have focussed on the direct problem about which sets admit minimal complements, see the works of Chen--Yang \cite{ChenYang12}, Kiss--S\'{a}ndor--Yang \cite{KissSandorYangJCT19}, of the authors \cite{MinComp1}, \cite{MinComp2} etc. Recently, the study of inverse problems, i.e., which sets occur as minimal complements, has become popular. The works of Kwon \cite{Kwon}, Alon--Kravitz--Larson \cite{AlonKravitzLarson}, Burcroff--Luntzlara \cite{BurcroffLuntzlara} and also of the authors \cite{CoMin1, CoMin2, CoMin3} have investigated this direction of research. However, most of the literature till date, has focussed on abelian groups. In this work, our motivation is two-fold:
\begin{enumerate}
	\item To show some new results on the inverse problem.
	\item To concentrate on the inverse problem in not necessarily abelian or finite groups.
\end{enumerate}

In \cite[Theorem C]{CoMin1}, it has been proved that the ``large'' subsets of a group cannot be a minimal complement to any subset. In \cite{AlonKravitzLarson}, Alon--Kravitz--Larson have established several interesting results which includes the above statement in the context of finite abelian groups. For any group $G$, \cite[Theorem C]{CoMin1} states that a subset $C$ of $G$, other than $G$, is not a minimal complement in $G$ if $C$ is ``large'' in the sense that 
\begin{equation}
\label{Eqn:BSRel}
\frac{|C| }{|G\setminus C|} > 2.
\end{equation}
In \cite[Theorem C]{CoMin1}, the set $G\setminus C$ was assumed to be finite. A refined version of this result in the context of finite abelian groups is established in  \cite[Proposition 17]{AlonKravitzLarson}, which states that a subset $C$ of a finite abelian group $G$, contained in a subgroup $H$, is not a minimal complement in $G$ if $C$ is ``large'' in the sense that 
$$
\frac{2|G||H| } {|H| + 2|G|}
< |C| 
< |H|.
$$
Note that the above inequality can be restated as 
\begin{equation}
\label{Eqn:AKLRel}
\frac{|C| }{|H\setminus C|} 
> 2[G:H]
\end{equation}
together with $C\subsetneq H$ (as explained in the proof of Proposition \ref{Prop:Fini}). 

We consider the subsets of $G$ which are contained in the subgroups of $G$ and establish a necessary condition (similar to Equations \eqref{Eqn:BSRel}, \eqref{Eqn:AKLRel}) for them to be  non-minimal complements in $G$. For a subset $C$ of $G$, strictly contained in a subgroup $H$, define the \textit{relative quotient of $C$ with respect to $H$} to be 
$$\lambda_H(C) 
= 
\frac{|C| }{|H\setminus C|} .$$
Note that \cite[Proposition 17]{AlonKravitzLarson} (in the context of finite abelian groups $G$), \cite[Theorem C]{CoMin1} (for any group $G$ with $G = H$) can be restated as follows: a subset $C$ of a group $G$, properly contained in a subgroup $H$ of $G$, is not a minimal complement in $G$ if its relative quotient with respect to $H$ is greater than the double of the index of $H$ in $G$, i.e., 
$$\lambda_H(C) > 2[G:H].$$
The aim of this article is to establish that such a statement holds in more general contexts. 

\subsection{Results obtained}
By suitably adapting the proof of \cite[Theorem C]{CoMin1}, we prove that a subset $C$ of a group $G$, properly contained in a subgroup $H$, is not a minimal complement in $G$ if the inequality 
$$\lambda_H(C) > 2[G:H]$$
holds (when the above inequality is interpreted in an appropriate manner). 
In fact, our results are more general. Under suitable hypothesis, we prove that not only such sets $C$, but also the sets of the form $(C\setminus E) \cup F$ are non-minimal complements for subsets $C$ of $H$ satisfying the above inequality, finite subsets $E\subseteq C$ and subsets $F\subseteq H\setminus C$. 
We refer to Theorems \ref{Thm:FAvoidsOneCoset}, \ref{Thm:QLeavesLAppears}, \ref{Thm:FContainedInSingleCoset}, \ref{Thm:CMinusCSymm}, \ref{Thm:Top}, \ref{Thm:Cardi} and Propositions \ref{Prop:Coset}, \ref{Prop:SansK}, \ref{Prop:Fini} for the precise statements. These results are more general than \cite[Proposition 17]{AlonKravitzLarson}, \cite[Theorem C]{CoMin1}. Using them, we obtain subsets of groups which are not minimal complements to any subset. Though the above-mentioned results apply to any group, to motivate the discussion, we provide the examples in the context of the integers. 

\begin{example}
\label{Eg:Intro}
\quad 
\begin{enumerate}
\item 
It follows from Theorem \ref{Thm:FAvoidsOneCoset} that the set 
$$(\{5, 7, \cdots, 27, 29\} + 32\bbZ)
\cup \{p\,|\, p \equiv \pm 1 \,(\mathrm{mod}\, 32), p \text{ is a prime}\}$$
is not a minimal complement in $\bbZ$. 

\item It follows from Theorem \ref{Thm:QLeavesLAppears} that the set 
$$(\{3, 9, 11, 13, \cdots, 47\} + 48\bbZ)
\cup \{p\,|\, p \equiv 1, 5, 7 \,(\mathrm{mod}\, 48), p \text{ is a prime}\}$$
is not a minimal complement in $\bbZ$. 

\item It follows from Theorem \ref{Thm:FContainedInSingleCoset} that the set 
$$(\{3, 5, 7, 9, 11\} + 12\bbZ)
\cup \{p\,|\, p \equiv 1 \,(\mathrm{mod}\, 12), p \text{ is a prime}\}$$
is not a minimal complement in $\bbZ$. 

\item 
It follows from Theorem \ref{Thm:CMinusCSymm} that the set 
$$(\{0, 1, 2, 3, 6, 7, 8\} + 9\bbZ) 
\cup \{p\,|\, p \equiv \pm 5 \,(\mathrm{mod}\, 9), p \text{ is a prime}\}
$$
is not a minimal complement in $\bbZ$.

\item 
It follows from Proposition \ref{Prop:Coset} that $\{2, 4, 6, 8, 10\} + 12\bbZ$ is not a minimal complement in $\bbZ$. Moreover, it also follows that the set of irrational numbers is not a minimal complement in $\bbR$, and the set of transcendental numbers is not a minimal complement in $\bbC$. 

\item 
It follows from Proposition \ref{Prop:SansK} that for any positive integer $k$ and for any nonempty finite subset $F$ of $k\bbZ$, the set $k\bbZ\setminus F$ is not a minimal complement in $\bbZ$. 

\item 
It follows from Theorem \ref{Thm:Top} that the set of real numbers having absolute value greater than one is not a minimal complement in $\bbR$. 

\item 
It follows from Theorem \ref{Thm:Cardi} that the set of irrational numbers, with a countable number of points removed, is not a minimal complement in $\bbR$, the set of transcendental numbers, with a countable number of points removed, is not a minimal complement in $\bbC$. 
\end{enumerate}
\end{example}

There are several immediate questions about the minimal complements in a finite group, for instance, given a group $G$ of order $n$, what are the sizes of the minimal complements, what are the integers $k$ between $1$ and $n$ such that any subset (or some subset) of $G$ of size $k$ is a minimal complement \cite[Question 1]{CoMin1}. Further, one can study these questions in the context of cyclic groups, or abelian groups, or finite groups. Some of these questions were answered by Alon, Kravitz and Larson in the context of abelian groups \cite[Theorem 1, Proposition 17]{AlonKravitzLarson}. The results obtained in Section \ref{Sec:NonMinComp} apply to groups, which are not assumed to be abelian, and thus they further improve our understanding about \cite[Question 1]{CoMin1}. 

Following \cite[Definition 5]{BurcroffLuntzlara}, one can consider the notion of robust MAC and robust non-MAC in any abelian group $G$. A subset of an abelian group $G$ is said to be a \textit{robust non-MAC} if it remains a non-minimal complement after the removal or the inclusion of finitely many points (see Definition \ref{Defn:Robust}). We obtain uncountably many examples of robust non-MACs in finitely generated abelian groups of positive rank and in any free abelian group of positive rank (see Theorem \ref{Thm:RobustNonMac} for a more general statement). Further, one can consider the analogous notion in non-abelian groups and obtain several examples by applying the results from Section \ref{Sec:NonMinComp}. In particular, we show that for any number field $K$ of degree $\geq 3$, the group $\gln_n(\calO_K)$ contains uncountably many robust non-minimal complements where $\calO_K$ denote the ring of integers of $K$. We refer to Section \ref{Sec:RobustNonMac} for the details. 

\section{Non-minimal complements in groups}
\label{Sec:NonMinComp}

The principal results of this Section are Theorems \ref{Thm:FAvoidsOneCoset}, \ref{Thm:QLeavesLAppears}, \ref{Thm:FContainedInSingleCoset}, \ref{Thm:CMinusCSymm}, \ref{Thm:Top}, \ref{Thm:Cardi}. They are aimed at establishing that a subset $C$ of a group $G$, properly contained in a subgroup $H$, is not a minimal complement in $G$ if the inequality 
$$\lambda_H(C) > 2[G:H]$$
holds (when the above inequality is interpreted in an appropriate manner). 
Moreover, these results not only deal with such sets $C$, but also deal with the sets of the form $(C\setminus E) \cup F$ where $C$ is a subset of $H$ satisfying the above inequality, $E$ is a finite subset of $C$ and $F\subseteq H\setminus C$. We refer to Theorems \ref{Thm:FAvoidsOneCoset}, \ref{Thm:QLeavesLAppears}, \ref{Thm:FContainedInSingleCoset}, \ref{Thm:CMinusCSymm}, \ref{Thm:Top}, \ref{Thm:Cardi} for the precise statements. These results are illustrated by applying them to subsets of certain groups, and thereby obtaining examples of non-minimal complements, see Remarks \ref{Remark:FAvoidsOneCoset}, \ref{Remark:QLeavesLAppears}, \ref{Remark:FContainedInSingleCoset}, \ref{Remark:CMinusCSymm}, \ref{Remark:Top}, \ref{Remark:Cardi}, see also Section \ref{Sec:RobustNonMac}. 
Some of their important consequences are stated in Propositions \ref{Prop:Coset}, \ref{Prop:SansK},  \ref{Prop:Fini}. 

We remark that no group is assumed to be abelian or finite unless otherwise stated. 

In the following, $H$ denotes a finite index subgroup of a group $G$, $K$ denotes a normal subgroup of $H$. If $X$ is a subset of $G$ and $X$ is the union of certain $K$-right cosets, then denote the number of $K$-right cosets contained in $X$ by $[X:K]$. Let $C$ denote a proper subset\footnote{A subset $A$ of a set $B$ is said to be a \textit{proper subset} if $B\setminus A$ is nonempty.} of $H$. Suppose $C$ is a union of certain right cosets of $K$ in $H$ and $H\setminus C$ is the union of finitely many right cosets of $K$ in $H$. Henceforth, we assume that the relative quotient of $C$ with respect to $H$ is greater than the double of the index of $H$ in $G$, i.e., the inequality 
$$\lambda_H(C) > 2[G:H]$$
holds in the following sense. 

\begin{assumption}
\label{Assumption}
The number of the $K$-right cosets contained in $C$ is greater than the product of $2[G:H]$ and the number of $K$-right cosets contained in $H\setminus C$.
\end{assumption}

Let $E$ be a finite subset of $C$ and $F$ be a subset of $H\setminus C$. 

\begin{theorem}
\label{Thm:FAvoidsOneCoset}
If 
\begin{enumerate}
\item the set $F$ does not intersect with some $K$-right coset in $H\setminus C$,
\item the number of elements of $K$ is greater than $2([G:H] + 1) |E|$,
\end{enumerate}
and Assumption \ref{Assumption} holds, then $(C\setminus E) \cup F$ is not a minimal complement in $G$.
\end{theorem}

\begin{proof}
On the contrary, let us assume that $(C\setminus E) \cup F$ is a minimal left complement to a subset $S$ of $G$. Let $\ell$ denote the index of $H$ in $G$. Let $s_1, \cdots, s_\ell$ be elements of $S$ such that 
$$H s_i \cap H s_j = \emptyset \quad \text{ for all } i \neq j.$$
For $1\leq i \leq \ell$, let $S_i$ denote the subset of $S$ defined by 
$$S_i : = 
\{s\in S\,|\, Hs = Hs_i\}.$$
By the first condition, it follows that $(C\setminus E) \cup F$ and $K\cdot ((C\setminus E) \cup F)$ are proper subsets of $H$. So, for each $1\leq i \leq \ell$, there exists an element $s_i'$ in $S_i$ such that 
$$
(K\cdot ((C\setminus E) \cup F))s_i' \neq (K\cdot ((C\setminus E) \cup F))s_i
$$
for all $1\leq i \leq \ell$. Since $K$ is normal in $H$, it follows that 
\begin{equation}
\label{Eqn:DistinctModK}
Ks_i \neq Ks_i'
\end{equation}
for any $i$. 

Note that there exists a subset $\calC$ of $C$ consisting of certain $K$-right cosets such that $\calC$ contains at most $\ell [(H \setminus C):K]$ many $K$-right cosets and $(\calC \cup F)\cdot S$ contains $(H\setminus C) \cdot \{s_1, \cdots, s_\ell\}$. Moreover, there exists a subset $\calE$ of $C\setminus E$ containing at most $|E|$ elements such that  $((\calC\setminus E) \cup \calE \cup F)\cdot S$ contains $(H\setminus C) \cdot \{s_1, \cdots, s_\ell\}$. Further, the set 
$$\calC \cup ((H\setminus C) \cdot  \{s_1's_1^\mo, \cdots, s_\ell' s_\ell^\mo\})$$
contains at most $2\ell  [(H \setminus C):K]$ many $K$-right cosets. By Assumption \ref{Assumption}, it follows that the set $C$ contains a $K$-right coset $Kh$ which is disjoint from the set 
$$\calC \cup ((H\setminus C) \cdot \{s_1's_1^\mo, \cdots, s_\ell's_\ell^\mo\}).$$
Note that there exists a subset $\calE'$ of $C\setminus E$ containing at most $\ell|E|$ elements such that $(\calE' \cup F) \cdot S$ contains $E\cdot \{s_1', \cdots, s_\ell'\}$. Further, note that there exists a subset $\calE''$ of $C\setminus E$ containing at most $\ell|E|$ elements such that $(\calE'' \cup F) \cdot S$ contains $E\cdot \{s_1, \cdots, s_\ell\}$.

We claim that 
$$Hs_i\subseteq 
(((C\setminus E) \setminus Kh) \cup \calE \cup \calE' \cup \calE'' \cup F )
\cdot S$$
for any $1\leq i \leq \ell$. Since $(\calC\setminus \calE) \cup F$ is contained in $(C\setminus E)\cup F$, the set $((\calC\setminus E)\cup \calE \cup F) \cdot S$ contains $(H\setminus C)\cdot s_i$ and $Kh$ does not intersect with $\calC\setminus E$, it follows that $(H\setminus C)\cdot s_i$ is contained in 
$$(((C\setminus E) \setminus Kh) \cup \calE \cup F )
\cdot S.$$
Note that $(C\setminus Kh)\cdot s_i$ is contained in 
$$(E\cdot s_i ) \cup 
\left( 
((C\setminus E) \setminus Kh)
\cdot S
\right),$$
which is contained in 
$$((C\setminus E) \setminus Kh) \cup \calE'' \cup F )
\cdot S.$$
Note that $Khs_i$ does not intersect with $(H\setminus C)\cdot \{s_i'\}$. Since $Hs_i = Hs_i'$, it follows that $Khs_i$ is contained in $C\cdot s_i'$. Further, note that $Khs_i$ does not intersect with $Khs_i'$, otherwise, $Khs_i = Khs_i'$. Since $K$ is normal in $H$, it follows that $hKs_i = hKs_i'$, which yields $Ks_i = Ks_i'$, contradicting $Ks_i \neq Ks_i'$. So $Khs_i$ is contained in $(C\setminus Kh)\cdot s_i'$. Since $(\calE' \cup F) \cdot S$ contains $E\cdot \{s_1', \cdots, s_\ell'\}$, it follows that $Khs_i$ is contained in 
$$(((C\setminus E) \setminus Kh) \cup \calE' \cup F )
\cdot S.$$
This proves the claim that 
$$Hs_i\subseteq 
(((C\setminus E) \setminus Kh) \cup \calE \cup \calE' \cup \calE'' \cup F )
\cdot S$$
for all $1\leq i\leq \ell$. So $((C\setminus E) \setminus Kh) \cup \calE \cup \calE' \cup \calE'' \cup F$ is a left complement to $S$. By the second condition, $((C\setminus E) \setminus Kh) \cup \calE \cup \calE' \cup \calE''$ is a proper subset of $C\setminus E$. Hence $(C\setminus E)\cup F$ is not a minimal left complement to $S$. 

If $(C\setminus E) \cup F$ is a minimal right complement to some subset $T$ of $G$, then $(C^\mo \setminus E^\mo) \cup F^\mo$ is a minimal left complement to $T^\mo$, which is impossible. 
\end{proof}

\begin{remark}
\label{Remark:FAvoidsOneCoset}
Taking $G = \bbZ, H = 2\bbZ, K = 32\bbZ$, $C = (\{5, 7, \cdots, 27, 29\} + 32\bbZ) -1$, 
and $F = \{p\,|\, p \equiv \pm 1  \,(\mathrm{mod}\, 32), p \text{ is a prime}\} -1$, it follows from Theorem \ref{Thm:FAvoidsOneCoset} that the set 
$$(\{5, 7, \cdots, 27, 29\} + 32\bbZ)
\cup \{p\,|\, p \equiv \pm 1  \,(\mathrm{mod}\, 32), p \text{ is a prime}\}$$
is not a minimal complement in $\bbZ$.
\end{remark}

In the proof of Theorem \ref{Thm:FAvoidsOneCoset}, Equation \eqref{Eqn:DistinctModK} played a crucial role. This equation was obtained by using the hypothesis that $F$ does not intersect with some $K$-right coset contained in $H\setminus C$. In the following result, we prove that even if $F$ intersects with each $K$-right coset contained in $H\setminus C$, one may obtain a similar result under an alternate hypothesis. 

\begin{theorem}
\label{Thm:QLeavesLAppears}
If 
\begin{enumerate}
\item $F$ is a proper subset of $H\setminus C$ and given $2[G:H]$ many elements $x_1, y_1, \cdots, x_{[G:H]} , y_{[G:H]}$ of $G$ with $x_i \neq y_i$ for any $i$, there exists a finite index subgroup $L$ of $K$ such that 
$Lx_i \neq Ly_i$ for any $i$ and $L$ is normal in $H$. 

\item for any finite index subgroup $L$ of $K$, the number of elements of $L$ is greater than $2([G:H] + 1) |E|$,
\end{enumerate}
and Assumption \ref{Assumption} holds, then $(C\setminus E) \cup F$ is not a minimal complement in $G$.
\end{theorem}

\begin{proof}
On the contrary, let us assume that $(C\setminus E) \cup F$ is a minimal left complement to a subset $S$ of $G$. Let $\ell$ denote the index of $H$ in $G$. Let $s_1, \cdots, s_\ell$ be elements of $S$ such that 
$$H s_i \cap H s_j = \emptyset \quad \text{ for all } i \neq j.$$
For $1\leq i \leq \ell$, let $S_i$ denote the subset of $S$ defined by 
$$S_i : = 
\{s\in S\,|\, Hs = Hs_i\}.$$
By the first condition, $(C\setminus E) \cup F$ is a proper subset of $H$. It follows that $S_i$ contains an element other than $s_i$. Let $1\leq i\leq \ell$ be an integer and $s_i'\neq s_i$ be an element of $S_i$. By the first condition, there exists a finite index subgroup $L$ of $K$ such that $L$ is normal in $H$ and $Ls_i \neq Ls_i'$ for any $i$. Replacing $K$ by $L$ (if necessary), we may (and do) assume that $Ks_i \neq Ks_i'$ for any $i$. 
Note that the same condition was obtained in Equation \eqref{Eqn:DistinctModK} in the course of the proof of Theorem \ref{Thm:FAvoidsOneCoset}. 
Proceeding in a similar fashion, we obtain the result. 
\end{proof}

\begin{corollary}
Suppose $C$ is a subset of $\bbZ$ and it is the union of translates of a nonzero subgroup $K$ of $\bbZ$. 
If $\lambda_\bbZ(C) > 2$, then 
$(C\setminus E) \cup F$ is not a minimal complement in $\bbZ$ for any finite subset $E$ of $C$ and for any proper subset $F$ of $\bbZ\setminus C$. 
\end{corollary}

\begin{remark}
\label{Remark:QLeavesLAppears}
Taking $G = \bbZ, H = 2\bbZ, K = 48\bbZ$, $C = \{2, 8, 10, 12, \cdots, 46\} + 48\bbZ$ and $F = \{p\,|\, p \equiv 1, 5, 7  \,(\mathrm{mod}\, 48), p \text{ is a prime}\}-1$, it follows from Theorem \ref{Thm:QLeavesLAppears} that the set 
$$(\{3, 9, 11, 13, \cdots, 47\} + 48\bbZ)
\cup \{p\,|\, p \equiv 1, 5, 7 \,(\mathrm{mod}\, 48), p \text{ is a prime}\}$$
is not a minimal complement in $\bbZ$. 
\end{remark}

Note that the proofs of Theorems \ref{Thm:FAvoidsOneCoset}, \ref{Thm:QLeavesLAppears} crucially relied on the observation that $S_i$ contains two elements which lies in two disjoint $K$-right cosets. However, if we consider a set of the form $(C\setminus E)\cup F$ and assume that it is a minimal left complement to some set $S$, it is not clear whether each $S_i$ has this property. We show in the following result that even in such a situation, one may obtain a similar result under an alternate hypothesis.  

\begin{theorem}
\label{Thm:FContainedInSingleCoset}
If 
\begin{enumerate}
\item the set $F$ is either empty or it is contained in a single $K$-right coset, 
\item 
the set $F\cup X$ does not contain any $K$-right coset for any subset $X$ of $H$ of size $\leq 2([G:H] + 1) |E|$,
\item the set $E \cdot F^\mo$ does not contain any $K$-right coset,
\end{enumerate}
and Assumption \ref{Assumption} holds, then $(C\setminus E) \cup F$ is not a minimal complement in $G$.
\end{theorem}

\begin{proof}
On the contrary, let us assume that $(C\setminus E) \cup F$ is a minimal left complement to a subset $S$ of $G$. Let $\ell$ denote the index of $H$ in $G$. Let $s_1, \cdots, s_\ell$ be elements of $S$ such that 
$$H s_i \cap H s_j = \emptyset \quad \text{ for all } i \neq j.$$
For $1\leq i \leq \ell$, let $S_i$ denote the subset of $S$ defined by 
$$S_i : = 
\{s\in S\,|\, Hs = Hs_i\}.$$
By the first and second condition, $(C\setminus E) \cup F$ is a proper subset of $H$. It follows that $S_i$ contains an element other than $s_i$. Let $P, Q$ denote the sets defined by 
\begin{align*}
P & 
= 
\{i \,|\, 1\leq i \leq \ell, Cs' \neq Cs_i \text{ for some } s'\in S_i\},\\
Q 
& = 
\{i \,|\, 1\leq i \leq \ell, i\notin P\}.
\end{align*}
For $i\in P$,  let $s_i'$ denote an element of $S_i$ such that $Cs_i' \neq Cs_i$. For $i\in Q$, let $s_i'$ denote an element of $S_i$ other than $s_i$. 

Note that there exists a subset $\calC$ of $C$ consisting of certain $K$-right cosets such that $\calC$ contains at most $\ell [(H \setminus C):K]$ many $K$-right cosets and $(\calC \cup F)\cdot S$ contains $(H\setminus C) \cdot \{s_1, \cdots, s_\ell\}$. There exists a subset $\calE$ of $C\setminus E$ containing at most $|E|$ elements such that $((\calC\setminus E) \cup \calE \cup F)\cdot S$ contains $(H\setminus C) \cdot \{s_1, \cdots, s_\ell\}$. Further, the set 
$$\calC \cup ((H\setminus C) \cdot  \{s_1's_1^\mo, \cdots, s_\ell' s_\ell^\mo\})$$
contains at most $2\ell  [(H\setminus C):K]$ many $K$-right cosets. 
By Assumption \ref{Assumption}, it follows that the set $C$ contains a $K$-right coset $\calR$ which is disjoint from the set 
$$\calC \cup ((H\setminus C) \cdot \{s_1's_1^\mo, \cdots, s_\ell's_\ell^\mo\}).$$
Note that there exists a subset $\calE'$ of $C\setminus E$ containing at most $\ell|E|$ elements such that $(\calE' \cup F) \cdot S$ contains $E\cdot \{s_1', \cdots, s_\ell'\}$. Further, note that there exists a subset $\calE''$ of $C\setminus E$ containing at most $\ell|E|$ elements such that $(\calE'' \cup F) \cdot S$ contains $E\cdot \{s_1, \cdots, s_\ell\}$.

Assume that $F$ is contained in $K\alpha$ for some $\alpha\in H$. Let $h$ be an element of $\calR$ such that $h\notin (E\cdot F^\mo ) \alpha$ (if there were no such $h$, then $(E\cdot F^\mo ) \alpha$ would contain $\calR$, which would imply that $E\cdot F^\mo$ contains a $K$-right coset, contradicting the third condition.). We claim that 
$$Hs_i\subseteq 
(((C\setminus E) \setminus Kh) \cup \calE \cup \calE' \cup \calE'' \cup (h\alpha^\mo F)\cup F )
\cdot S$$
for any $1\leq i \leq \ell$. Since $(\calC\setminus \calE) \cup F$ is contained in $(C\setminus E)\cup F$, the set $((\calC\setminus E)\cup \calE \cup F) \cdot S$ contains $(H\setminus C)\cdot s_i$ and $Kh$ does not intersect with $\calC\setminus E$, it follows that $(H\setminus C)\cdot s_i$ is contained in 
$$(((C\setminus E) \setminus Kh) \cup \calE \cup F )
\cdot S.$$
Note that $(C\setminus Kh)\cdot s_i$ is contained in 
$$(E\cdot s_i ) \cup 
\left( 
(((C\setminus E) \setminus Kh) \cup F )
\cdot S
\right),$$
which is contained in 
$$((C\setminus E) \setminus Kh) \cup \calE'' \cup F )
\cdot S.$$

Let $i$ be an element of $P$. Note that $Khs_i$ does not intersect with $(H\setminus C)\cdot \{s_i'\}$. Since $Hs_i = Hs_i'$, it follows that $Khs_i$ is contained in $C\cdot s_i'$. Further, note that $Khs_i$ does not intersect with $Khs_i'$, otherwise, $Khs_i = Khs_i'$. Since $K$ is normal in $H$, it follows that $hKs_i = hKs_i'$, which yields $Ks_i = Ks_i'$, and consequently 
$K\widetilde h s_i' = K\widetilde h s_i$
holds for any $\widetilde h\in H$, contradicting $i\in P$. So $Khs_i$ is contained in $(C\setminus Kh)\cdot s_i'$. Since $(\calE' \cup F) \cdot S$ contains $E\cdot \{s_1', \cdots, s_\ell'\}$, it follows that $Khs_i$ is contained in 
$$((C\setminus E) \setminus Kh) \cup \calE' \cup F )
\cdot S$$
for $i\in P$. 

For each $i\in Q$, we have  
$$F\cdot S_i = (H \setminus C) \cdot s_i.$$
Note that for any $\beta\in H$, we obtain 
\begin{align*}
\beta \alpha^\mo F 
& \subseteq 
\beta \alpha^\mo K \alpha \\
& = 
K \beta \alpha^\mo \alpha \\
& = 
K \beta
\end{align*}
and 
\begin{align*}
(\beta \alpha^\mo F )\cdot S_i 
& = 
(\beta \alpha^\mo) \cdot (F \cdot S_i) \\
& \supseteq 
(\beta \alpha^\mo) \cdot K \alpha s_i \\
& = 
K \beta \alpha^\mo  \alpha s_i\\
& = 
K \beta s_i
\end{align*}
for $i\in Q$. It follows that $Khs_i$ is contained in $(h\alpha^\mo F)\cdot S$ for $i\in Q$. 

This proves the claim that 
$$Hs_i\subseteq 
((C\setminus E) \setminus Kh) \cup \calE \cup \calE' \cup \calE'' \cup (h\alpha^\mo F) \cup F )
\cdot S$$
for all $i$. So $(C\setminus E) \setminus Kh) \cup \calE \cup \calE' \cup \calE'' \cup (h\alpha^\mo F) \cup F$ is a left complement to $S$. Since $h\notin (E\cdot F^\mo ) \alpha$, using the second condition, $((C\setminus E) \setminus Kh) \cup \calE \cup \calE' \cup \calE'' \cup (h\alpha^\mo F)$ is a proper subset of $C\setminus E$. Hence $(C\setminus E)\cup F$ is not a minimal left complement to $S$. 

If $(C\setminus E) \cup F$ is a minimal right complement to some subset $T$ of $G$, then $(C^\mo \setminus E^\mo) \cup F^\mo$ is a minimal left complement to $T^\mo$, which is impossible. 
\end{proof}

\begin{remark}
\label{Remark:FContainedInSingleCoset}
Taking $G = \bbZ, H = 2\bbZ, K = 12\bbZ$, $C = \{2, 4, 6, 8, 10\} + 12\bbZ$ and $F = \{p\,|\, p \equiv 1  \,(\mathrm{mod}\, 12), p \text{ is a prime}\}-1$, it follows from Theorem \ref{Thm:FContainedInSingleCoset} that the set 
$$(\{3, 5, 7, 9, 11\} + 12\bbZ)
\cup \{p\,|\, p \equiv 1  \,(\mathrm{mod}\, 12), p \text{ is a prime}\}$$
is not a minimal complement in $\bbZ$. 
\end{remark}

Note that in the proof of Theorem \ref{Thm:FContainedInSingleCoset}, the hypothesis that $F$ is either empty or is contained in a single $K$-right coset, played a crucial role. It would be interesting to consider the subsets of $H$ of the form $(C\setminus E) \cup F$ for ``large'' $C$ and for any proper subset $F$ of $H\setminus C$. In Theorem \ref{Thm:CMinusCSymm}, we prove that even if $F$ intersects with each $K$-right coset contained in $H\setminus C$, one may obtain a similar result under an alternate hypothesis.

\begin{proposition}
\label{Prop:CMinusCSymmPrimeIndex}
Let $X, Y$ be two nonempty disjoint subsets of a group $G$ with $X\cup Y = G$. Let $L$ be a subgroup of $G$ such that $X$ is the union of certain right cosets of $L$. Then the inclusion 
\begin{equation}
\label{Eqn:CondXYL}
(X\cdot Y^\mo) \cup (Y \cdot X^\mo) 
\subseteq G\setminus L
\end{equation}
holds. 
Moreover, the following conditions are equivalent. 
\begin{enumerate}
\item The inclusion in Equation \eqref{Eqn:CondXYL}
is a proper inclusion.
\item 
For each $y\in Y$, there exists an element $y'\in Y\setminus (Ly)$ such that
$$y'y^\mo \cdot Y = Y.$$

\item 
For some $y\in Y$, there exists an element $y'\in Y\setminus (Ly)$ such that
$$y'y^\mo \cdot Y = Y.$$
\item 
The set $Y$ is the union of certain right cosets of some subgroup of $G$ which properly contains $L$. 
\item 
The set $X$ is the union of certain right cosets of some subgroup of $G$ which properly contains $L$. 
\end{enumerate}
The set
$$
G\setminus 
\left((X\cdot Y^\mo) \cup (Y \cdot X^\mo) \right)$$
is a subgroup of $G$ and it is the maximal subgroup of $G$ such that $Y$ is a union of its right cosets. 
If $Y$ is finite, then the inclusion
\begin{equation}
\label{Eqn:CondXY}
(X\cdot Y^\mo) \cup (Y \cdot X^\mo) 
\subseteq G\setminus \{e\}
\end{equation}
is an equality under any one of the following conditions. 
\begin{enumerate}[(a)]
\item The order of $y'y^\mo$ is greater than the size of $Y$ for any $y, y'\in Y$ with $y \neq y'$.
\item 
The size of $Y$ is not divisible by the size of any nontrivial finite subgroup of $G$.
\end{enumerate}
\end{proposition}

\begin{proof}
Since $X, Y$ are disjoint and each of them can be expressed as the union of certain $L$-right cosets, it follows that the inclusion in Equation \eqref{Eqn:CondXYL} holds. 

Note that 
$$
((X\cdot g)\cdot (Y\cdot g)^\mo) \cup ((Y\cdot g) \cdot (X\cdot g)^\mo) 
 = 
(X\cdot Y^\mo) \cup (Y \cdot X^\mo) 
$$
for any $g\in G$. Thus the inclusion
$$
(X\cdot Y^\mo) \cup (Y \cdot X^\mo) 
\subseteq G\setminus L
$$
is an equality if and only if the inclusion 
$$
((X\cdot g)\cdot (Y\cdot g)^\mo) \cup ((Y\cdot g) \cdot (X\cdot g)^\mo) 
\subseteq G\setminus L
$$
is an equality. 

Note that for $y, y'\in Y$, the element $y'y^\mo $ does not belong to the set $((X\cdot y^\mo)\cdot (Y\cdot y^\mo)^\mo) \cup ((Y\cdot y^\mo) \cdot (X\cdot y^\mo)^\mo) $ if and only if 
$$y'y^\mo \cdot (Y\cdot y^\mo) 
\subseteq 
Y\cdot y^\mo, 
\quad \text{ and } \quad 
y'y^\mo \cdot (X\cdot y^\mo) 
\subseteq 
X \cdot y^\mo,
$$
which holds if and only if 
$$y'y^\mo \cdot Y = Y.$$

Assume that the first condition holds. Choose an element $y\in Y$. Note that the set $((X\cdot y^\mo)\cdot (Y\cdot y^\mo)^\mo) \cup ((Y\cdot y^\mo) \cdot (X\cdot y^\mo)^\mo) $ contains $X\cdot y^\mo$. So this set does not contain $y'y^\mo $ for some $y'\in Y\setminus (Ly)$. We obtain 
$$y'y^\mo \cdot Y = Y.$$
So the first condition implies the second condition. Note that the second condition implies the third condition. Now, assume that the third condition holds, i.e.,
$$y'y^\mo \cdot Y = Y$$
holds with $y, y'\in Y$, $Ly \neq Ly'$. Let $L'$ denote the subgroup of $G$ generated by $L$ and $y'y^\mo$. Since $x\cdot Y = Y$ for any $x\in L$, and $y'y^\mo \cdot Y = Y$, it follows that $x\cdot Y = Y$ for any $x\in L'$. So, $Y$ is union of certain right cosets of $L'$. Since $Ly \neq Ly'$, it follows that $L$ is properly contained in $L'$. Thus the fourth condition follows. Assume that the fourth condition holds, i.e., $Y$ is the union of certain translates of some subgroup $\calL$ of $G$ which properly contains $L$. Then the set $
(X\cdot Y^\mo) \cup (Y \cdot X^\mo) $ does not contain $\calL$. Since $\calL$ properly contains $L$, the first condition follows. 
Since $X, Y$ are disjoint and $X \cup Y = G$, the fourth and the fifth conditions are equivalent. 
This proves the equivalence of the five conditions.

Consider the subgroups $L'$ of $G$ such that $L'$ contains $L$ and $Y$ can be expressed as the union of right cosets of $L'$. Let $\scrL$ denote the subgroup of $G$ generated by such subgroups. Note that $Y$ can be expressed as the union of the right cosets of $\scrL$. It follows that 
$$(X\cdot Y^\mo) \cup (Y \cdot X^\mo) 
\subseteq G\setminus \scrL.$$
By the construction of $\scrL$, it follows that the above inclusion is an equality, and it also follows that $\scrL$ is the maximal subgroup of $G$ such that $Y$ is a union of its right cosets.

Suppose $Y$ is finite and the order of $y'y^\mo$ is greater than the size of $Y$ for any $y, y'\in Y$ with $y \neq y'$. 
Assume that the inclusion in Equation \eqref{Eqn:CondXY} is not an equality. So, there exist two distinct elements $y_1, y_2 \in Y$ such that 
$$y_1y_2^\mo \cdot Y = Y.$$
Let $y_0$ be an element of $Y$. Denote the order of $y_1 y_2^\mo$ by $r$. Then the set $Y$ contains the $r$ distinct elements $yy_0, y^2y_0, y^3y_0, \cdots, y^ry_0$ where $y = y_1 y_2^\mo$, which is impossible, since $r$ is greater than the size of $Y$. Hence, the inclusion in Equation \eqref{Eqn:CondXY} is an equality. Moreover, if $Y$ is finite and the size of $Y$ is not divisible by the size of any nontrivial finite subgroup of $G$, then $Y$ cannot be expressed as the union of certain right cosets of some nontrivial subgroup of $G$. Hence, the inclusion in Equation \eqref{Eqn:CondXY} is an equality. 
\end{proof}

\begin{theorem}
\label{Thm:CMinusCSymm}
If 
\begin{enumerate}
\item 
\begin{equation}
\label{Eqn:Cond}
(C^\mo(H\setminus C))
\cup 
((H\setminus C)^\mo C)
= H \setminus K
\end{equation}
\item 
the set $F\cup X$ does not contain a $K$-right coset for any subset $X$ of $H$ of size $\leq 2([G:H] + 1) |E|$,
\item the set $E \cdot F^\mo$ does not contain any $K$-right coset,
\end{enumerate}
and Assumption \ref{Assumption} holds, then $(C\setminus E) \cup F$ is not a minimal complement in $G$.
In particular, $(C\setminus E) \cup F$ is not a minimal complement in $G$ if 
\begin{enumerate}
\item 
any subgroup $K'$ of $H$ such that $K$ is contained in $K'$ and $H\setminus C$ can be expressed as a union of right cosets of $K'$, is normal in $H$, 

\item 
the set $F\cup X$ does not contain a $K$-right coset for any subset $X$ of $H$ of size $\leq 2([G:H] + 1) |E|$,
\item the set $E \cdot F^\mo$ does not contain any $K$-right coset,
\end{enumerate}
and Assumption \ref{Assumption} holds.
\end{theorem}

\begin{proof}
On the contrary, let us assume that $(C\setminus E) \cup F$ is a minimal left complement to a subset $S$ of $G$. Let $\ell$ denote the index of $H$ in $G$. Let $s_1, \cdots, s_\ell$ be elements of $S$ such that 
$$H s_i \cap H s_j = \emptyset \quad \text{ for all } i \neq j.$$
For $1\leq i \leq \ell$, let $S_i$ denote the subset of $S$ defined by 
$$S_i : = 
\{s\in S\,|\, Hs = Hs_i\}.$$
By the second condition, $(C\setminus E) \cup F$ is a proper subset of $H$. It follows that $S_i$ contains an element other than $s_i$. Let $P, Q$ denote the sets defined by 
\begin{align*}
P & 
= 
\{i \,|\, 1\leq i \leq \ell, Cs' \neq Cs_i \text{ for some } s'\in S_i\},\\
Q 
& = 
\{i \,|\, 1\leq i \leq \ell, i\notin P\}.
\end{align*}
For $i\in P$,  let $s_i'$ denote an element of $S_i$ such that $Cs_i' \neq Cs_i$. For $i\in Q$, let $s_i'$ denote an element of $S_i$ other than $s_i$. 

Note that there exists a subset $\calC$ of $C$ consisting of certain $K$-right cosets such that $\calC$ contains at most $\ell [(H \setminus C):K]$ many $K$-right cosets and $(\calC \cup F)\cdot S$ contains $(H\setminus C) \cdot \{s_1, \cdots, s_\ell\}$. Moreover, there exists a subset $\calE$ of $C\setminus E$ containing at most $|E|$ elements such that $((\calC\setminus E) \cup \calE \cup F)\cdot S$ contains $(H\setminus C) \cdot \{s_1, \cdots, s_\ell\}$. Further, the set 
$$\calC \cup ((H\setminus C) \cdot  \{s_1's_1^\mo, \cdots, s_\ell' s_\ell^\mo\})$$
contains at most $2\ell  [(H\setminus C):K]$ many $K$-right cosets. By Assumption \ref{Assumption}, it follows that the set $C$ contains a $K$-right coset class $\calR$ which is disjoint from the set 
$$\calC \cup ((H\setminus C) \cdot \{s_1's_1^\mo, \cdots, s_\ell's_\ell^\mo\}).$$
Note that there exists a subset $\calE'$ of $C\setminus E$ containing at most $\ell|E|$ elements such that $(\calE' \cup F) \cdot S$ contains $E\cdot \{s_1', \cdots, s_\ell'\}$. Further, note that there exists a subset $\calE''$ of $C\setminus E$ containing at most $\ell|E|$ elements such that $(\calE'' \cup F) \cdot S$ contains 
$E\cdot \{s_1, \cdots, s_\ell\}$.

Assume that $F\cap K\alpha$ is properly contained in $K\alpha$ for some $\alpha\in H$. Let $h$ be an element of $\calR$ such that $h\notin (E\cdot F^\mo ) \alpha$ (if there were no such $h$, then $(E\cdot F^\mo ) \alpha$ would contain $\calR$, which would imply that $E\cdot F^\mo$ contains a $K$-right coset, contradicting the third condition.). We claim that 
$$Hs_i\subseteq 
(((C\setminus E) \setminus Kh) \cup \calE \cup \calE' \cup \calE'' \cup (h\alpha^\mo F\cap Kh)\cup F )
\cdot S$$
for any $1\leq i \leq \ell$. Since $(\calC\setminus \calE) \cup F$ is contained in $(C\setminus E)\cup F$, the set $((\calC\setminus E)\cup \calE \cup F) \cdot S$ contains $(H\setminus C)\cdot s_i$ and $Kh$ does not intersect with $\calC\setminus E$, it follows that $(H\setminus C)\cdot s_i$ is contained in 
$$((C\setminus E) \setminus Kh) \cup \calE \cup F )
\cdot S.$$
Note that $(C\setminus Kh)\cdot s_i$ is contained in 
$$(E\cdot s_i ) \cup 
\left( 
(((C\setminus E) \setminus Kh) \cup F )
\cdot S
\right),$$
which is contained in 
$$((C\setminus E) \setminus Kh) \cup \calE'' \cup F )
\cdot S.$$

Let $i$ be an element of $P$. Note that $Khs_i$ does not intersect with $(H\setminus C)\cdot \{s_i'\}$. Since $Hs_i = Hs_i'$, it follows that $Khs_i$ is contained in $C\cdot s_i'$. Further, note that $Khs_i$ does not intersect with $Khs_i'$, otherwise, $Khs_i = Khs_i'$. Since $K$ is normal in $H$, it follows that $hKs_i = hKs_i'$, which yields $Ks_i = Ks_i'$, and consequently 
$K\widetilde h s_i' = K\widetilde h s_i$
holds for any $\widetilde h\in H$, contradicting $i\in P$. So $Khs_i$ is contained in $(C\setminus Kh)\cdot s_i'$. Since $(\calE' \cup F) \cdot S$ contains $E\cdot \{s_1', \cdots, s_\ell'\}$, it follows that $Khs_i$ is contained in 
$$((C\setminus E) \setminus Kh) \cup \calE' \cup F )
\cdot S$$
for $i\in P$. 

We choose an element $i\in Q$. Note that the set $S_i$ is contained in $Ks_i$. Otherwise, there exists an element $s_i'\in S_i$ such that $K s_i ' \neq Ks_i$. Let $\beta$ denote the element  $s_i's_i^\mo$ of $H$. Note that $\beta$ does not lie in $K$. By the first condition (i.e. Equation \eqref{Eqn:Cond}), $\beta^{\pm 1} = \gamma ^\mo \delta$ with $\gamma \in C, \delta \in H\setminus C$. If $\beta = \gamma ^\mo \delta$, then $C\beta$ intersects with $H\setminus C$, and hence $C\beta s_i $ intersects with $(H\setminus C)s_i$, i.e., $C s_i' $ intersects with $(H\setminus C)s_i$, which implies that $Cs_i \neq Cs_i'$. If $\beta^\mo = \gamma ^\mo \delta$, then $C\beta^\mo$ intersects with $H\setminus C$, and hence $C\beta^\mo s_i' $ intersects with $(H\setminus C)s_i'$, i.e., $C s_i $ intersects with $(H\setminus C)s_i'$, which implies that $Cs_i '\neq Cs_i$. This shows that $i\in P$, which is a contradiction. So $S_i$ is contained in $Ks_i$. Since $K\alpha s_i$ is contained in $(C\cup F)\cdot S_i$, it follows that $K\alpha s_i$ is contained in $(F\cap K\alpha)\cdot S_i$, and hence 
\begin{align*}
(h\alpha^\mo F \cap Kh)\cdot S_i 
& = 
(h\alpha^\mo (F \cap (\alpha h^\mo Kh)))\cdot S_i \\
& = 
(h\alpha^\mo (F \cap (K \alpha h^\mo h)))\cdot S_i \\
& = 
(h\alpha^\mo (F \cap K \alpha))\cdot S_i \\
& = 
(h\alpha^\mo) ((F \cap K \alpha)\cdot S_i) \\
& \supseteq 
(h\alpha^\mo)  K \alpha s_i \\
& = 
K (h\alpha^\mo)  \alpha s_i \\
& = 
K h s_i.
\end{align*}
It follows that $Khs_i$ is contained in $(h\alpha^\mo F\cap Kh)\cdot S$ for any $i\in Q$. 

This proves the claim that 
$$Hs_i\subseteq 
((C\setminus E) \setminus Kh) \cup \calE \cup \calE' \cup \calE'' \cup (h\alpha^\mo F\cap Kh) \cup F )
\cdot S$$
for all $1\leq i\leq \ell$. So $(C\setminus E) \setminus Kh) \cup \calE \cup \calE' \cup \calE'' \cup (h\alpha^\mo F\cap Kh) \cup F$ is a left complement to $S$. By the second condition, $((C\setminus E) \setminus Kh) \cup \calE \cup \calE' \cup \calE'' \cup (h\alpha^\mo F\cap Kh)$ is a proper subset of $C\setminus E$. Hence $(C\setminus E)\cup F$ is not a minimal left complement to $S$. 

If $(C\setminus E) \cup F$ is a minimal right complement to some subset $T$ of $G$, then $(C^\mo \setminus E^\mo) \cup F^\mo$ is a minimal left complement to $T^\mo$, which is impossible. 

Now we establish the second part. 
Since $H\setminus C$ is the union of finitely many right cosets of $K$, it follows from Proposition \ref{Prop:CMinusCSymmPrimeIndex} that
$$(C^\mo(H\setminus C))
\cup 
((H\setminus C)^\mo C)
= H \setminus K'$$
and $H\setminus C$ is the union of certain left cosets of $K'$ for some subgroup $K'$ of $H$ containing $K$ as a finite index subgroup. 
Since $K'$ contains $K$, it follows from the hypothesis that the set $F\cup X$ does not contain a $K'$-right coset for any subset $X$ of $H$ of size $\leq 2([G:H] + 1) |E|$, and the set $E \cdot F^\mo$ does not contain any $K'$-right coset. Hence from the first part, the result follows. 
\end{proof}

\begin{remark}
By Proposition \ref{Prop:CMinusCSymmPrimeIndex}, the first condition in Theorem \ref{Thm:CMinusCSymm} is equivalent to requiring that $H\setminus C$ cannot be expressed (or equivalently, $C$ cannot be expressed) as the union of certain left cosets of some subgroup $L$ of $H$ satisfying $L\supsetneq K$.
\end{remark}

\begin{remark}
\label{Remark:CMinusCSymm}
Taking $G = \bbZ, H = 2\bbZ, K = 2n\bbZ$, it follows from Theorem \ref{Thm:CMinusCSymm} that for any integer $n\geq 11$, for any $1\leq a < b \leq n$ with 
\begin{equation}
\label{Eqn:CongruenceOdd}
2(a-b) \not\equiv 0\pmod n, 
\end{equation}
the set 
$$((\{2, 4, 6, \cdots, 2n\} \setminus \{2a, 2b\}) + 2n\bbZ)
\cup 
F$$
is not a minimal complement in $\bbZ$ for any proper subset of $F$ of $\{2a, 2b\} + 2n\bbZ$ since 
\begin{align*}
& ((\{2, 4, 6, \cdots, 2n\} \setminus \{2a, 2b\}) + 2n\bbZ)+ (\{-2a,-2b\} + 2n\bbZ)\\
& = ((\{2, 4, 6, \cdots, 2n\} \setminus \{0, 2(b-a)\}) + 2n\bbZ)
\cup 
((\{2, 4, 6, \cdots, 2n\} \setminus \{2(a-b), 0\})  + 2n\bbZ) \\
& 
= \{2, 4, 6, \cdots, 2n\} \setminus \{0\} + 2n\bbZ.
\end{align*}
Note that Equation \eqref{Eqn:CongruenceOdd} holds when $n$ is odd. One can obtain a more general example than the above. Taking $G = \bbZ, H = 2\bbZ, K = 2n\bbZ$, it follows from Proposition \ref{Prop:CMinusCSymmPrimeIndex} and Theorem \ref{Thm:CMinusCSymm} that for any integer $k\geq 2$, $n\geq 5k + 1$ such that $n$ is not divisible by any integer $1< i \leq k$ and for any $1\leq a_1 < a_2 < \cdots < a_k  \leq n$, 
the set 
$$((\{2, 4, 6, \cdots, 2n\} \setminus \{2a_1,2a_2, \cdots, 2a_k\}) + 2n\bbZ)
\cup 
F$$
is not a minimal complement in $\bbZ$ for any proper subset of $F$ of $\{2a_1,2a_2, \cdots, 2a_k\} + 2n\bbZ$. 
\end{remark}

When $F$ is the empty set, Theorems \ref{Thm:FAvoidsOneCoset}, \ref{Thm:FContainedInSingleCoset} are equivalent. One obtains the following consequences. 

\begin{proposition}
\label{Prop:Coset}
If 
\begin{enumerate}
\item the number of elements of $K$ is greater than $2([G:H] + 1) |E|$,
\end{enumerate}
and Assumption \ref{Assumption} holds, then $C\setminus E$ is not a minimal complement in $G$.
\end{proposition}

\begin{remark}
\label{Remark:Coset}
It follows from Proposition \ref{Prop:Coset} that $\{2, 4, 6, 8, 10\} + 12\bbZ$ is not a minimal complement in $\bbZ$. Taking $G = H = \bbR, K = \mathbb Q$, it follows from Proposition \ref{Prop:Coset} that the set of irrational numbers is not a minimal complement in $\bbR$. Taking $G = H = \bbC, K = \overline{\mathbb Q}$, it follows from Proposition \ref{Prop:Coset} that the set of transcendental numbers is not a minimal complement in $\bbC$. 
\end{remark}

When $K$ is the trivial subgroup and $E$ is the emptyset in Proposition \ref{Prop:Coset}, one obtains the following results. 

\begin{proposition}
\label{Prop:SansK}
If $H\setminus C$ is finite and $C$ contains more than $2[G:H] |H\setminus C|$ elements, i.e., the relative quotient of $C$ with respect to $H$ satisfies 
$$\lambda_H(C) > 2[G:H],$$
then $C$ is not a minimal complement to any subset of $G$. In particular, if $D$ is a proper subset of $G$ such that $G \setminus D$ is finite and $D$ contains more than $2|G\setminus D|$ elements, then $D$ is not a minimal complement to any subset of $G$. 
\end{proposition}

\begin{proof}
The first part follows from Proposition \ref{Prop:Coset}. The second past follows from the first part as a consequence. 
\end{proof}

\begin{remark}
It follows from Proposition \ref{Prop:SansK} that for any positive integer $k$ and for any nonempty finite subset $F$ of $k\bbZ$, the set $k\bbZ\setminus F$ is not a minimal complement in $\bbZ$. 
\end{remark}

In the context of finite groups, one has the following consequence of Proposition \ref{Prop:SansK}. See also \cite[Proposition 17]{AlonKravitzLarson}. 

\begin{proposition}
\label{Prop:Fini}
If $G$ is finite and the relative quotient of $C$ with respect to $H$ satisfies 
$$\lambda_H(C) > 2[G:H],$$
i.e., $C$ is a subset of $H$ satisfying 
$$ |H| > |C| > 2[G:H] |H\setminus C|,$$
then $C$ is not a minimal complement to any subset of $G$. Equivalently, no subset $C$ of a subgroup $H$ of a finite group $G$ satisfying 
$$
|H| \frac{2[G:H]} {1 + 2[G:H]}
= 
\frac{2|G||H| } {|H| + 2|G|}
< |C| 
< |H|
$$
is a minimal complement to some subset of $G$. In particular, if $C$ is a proper subset of a finite group $G$ containing more than $2|G\setminus C|$ elements, then $C$ is not a minimal complement in $G$. 
\end{proposition}

\begin{proof}
The first statement and the third statement follow from Proposition \ref{Prop:SansK}. 

To obtain the second statement, note that for a subset $C$ of $H$, the inequality $|C| > 2[G:H] |H\setminus C|$ is equivalent to 
$$
(2[G:H] + 1)|C| 
> 
2[G:H] |C| + 2[G:H] |H\setminus C| 
= 2[G:H] |H|
= 2|G| 
,$$
which is equivalent to 
$$|C| > \frac{2|G||H| } {|H| + 2|G|}.$$
Then the second part follows from the first part. 
\end{proof}

\begin{remark}
It follows from Proposition \ref{Prop:Fini} that the set $\{\overline 2, \overline 4, \overline 6, \overline 8, \overline{10}\}$ is not a minimal complement in $\bbZ/12\bbZ$. 
\end{remark}

\begin{theorem}
\label{Thm:Top}
Let $\calH $ be a finite index subgroup of a topological group $\calG $. Let $\calC $ be a proper subset of $\calH $. Suppose $\calH \setminus \calC $ is compact and closed\footnote{Note that the compact subsets need not be closed unless the ambient topological space is assumed to be Hausdorff.} in $\calH $. If $\calH $ is not a union of finitely many translates of $\calH \setminus \calC $, then $\calC $ is not a minimal complement in $\calG$. In particular, if $\calC $ is a proper subset of a topological group of $\calG $ such that $\calG \setminus \calC $ is closed and compact, and $\calG \setminus \calC $ is ``small'' in the sense that $\calG $ is not a union of finitely many translates of $\calG \setminus \calC $, then $\calC $ is not minimal complement in $\calG $. 
\end{theorem}

\begin{proof}
On the contrary, let us assume that $\calC $ is a minimal left complement to some subset $T$ of $\calG $. Let $S$ be a subset of $T$ such that $\calH \cdot S = \calG $, and $\calH  s_1 \cap \calH  s_2 = \emptyset$ for any two distinct elements $s_1, s_2\in S$. Since $\calC$ is a proper subset of $\calH $, for each $s\in S$, there exists an element $t_s\in T$ such that $t_s \neq s$ and $\calH  s = \calH  t_s$. 
Since $\scrC $ is compact and $S$ is finite, there is a nonempty finite subset $T'$ of $T$ such that $\{\calC \cdot s\}_{s\in T'}$ is an open cover of $\scrC \cdot S$. From the hypothesis, it follows that the subgroup $\calH $ strictly contains 
$$(\calH  \cap (\cup_{x\in S\cdot T'^\mo} \scrC \cdot x)) \cup (\cup_{s\in S} \scrC \cdot t_ss^\mo ) \cup \scrC ,$$
and hence there is an element $h\in \calH $ lying outside this union. We claim that $\calC \setminus\{h\}$ is a left complement to $T$. Note that $\calC \setminus \{h\}$ is nonempty. It suffices to show that $\calH s$ is contained in $(\calC \setminus \{h\})\cdot T$ for each $s\in S$. Let $k$ be an element of $\scrC $. Then $ks$ is equal to $ct'$ for some $c\in \calC , t'\in T'$. So $c$ lies in the above union and hence $h \neq c$. Thus $ks$ lies in $(\calC \setminus \{h\})\cdot T$. So $\scrC \cdot s$ is contained in $(\calC \setminus \{h\})\cdot T$. Note that $hs$ lies in $\calH t_s$ and does not lie in $\scrC t_s$. Thus $hs$ lies in $\calC \cdot t_s$. Since $s\neq t_s$, it follows that $hs$ lies in $(\calC \setminus \{h\})\cdot t_s$. Clearly, $(\calC \setminus \{h\}) s$ is contained in $(\calC \setminus \{h\})\cdot T$. So $\calH \cdot s$ is contained in $(\calC \setminus \{h\})\cdot T$. Thus $\calC \setminus \{h\}$ is a minimal left complement to $T$. Note that $h$ lies in $\calH $ and does not lie in $\scrC $. So $\calC \setminus \{h\}$ is a proper subset of $\calC $. Hence $\calC $ is not a minimal left complement to $T$. Similarly, assuming $\calC $ to be a right minimal complement to some subset of $\calG $ will lead to a contradiction. Hence $\calC $ is not a minimal complement in $\calG $. 

The second statement follows from the first statement. 
\end{proof}

\begin{remark}
\label{Remark:Top}
From Theorem \ref{Thm:Top}, it follows that the set of real numbers having absolute value greater than one is not a minimal complement in $\bbR$. 
\end{remark}

\begin{corollary}
Let $H$ be a finite index subgroup of an infinite group and $C$ be a proper subset of $H$ such that $H\setminus C$ is finite. 
Then $C$ is not a minimal complement in $G$. 
\end{corollary}

\begin{proof}
If $G$ is endowed with the discrete topology, then this corollary follows from Theorem \ref{Thm:Top}. 

It can also be seen as an immediate consequence of Proposition \ref{Prop:SansK} (and also of Proposition \ref{Prop:Coset}).
\end{proof}

\begin{corollary}
For any positive integer $k$ and for any nonempty finite subset $F$ of $k\bbZ$, the set $k\bbZ\setminus F$ is not a minimal complement in $\bbZ$. 
\end{corollary}

It turns out that the set of irrational numbers is not a minimal complement in $\bbR$ (see Remark \ref{Remark:Coset}). 
It is natural to ask whether the set of irrational numbers, with a countable number of points removed, is a minimal complement in $\bbR$. Theorem \ref{Thm:Top} does not seem to shed any light on this question since the set of irrational numbers, with a countable number of points removed, does not form a closed or a compact set under the Euclidean topology. 

\begin{theorem}
\label{Thm:Cardi}
Let $\calH$ be a subgroup of a group $\calG$. Let $\calC$ be a proper subset of $\calH$. Suppose Assumption \ref{Assumption} holds in the sense that no map from $\{0, 1\} \times (\calH\setminus \calC) \times (\calG/\calH)$ to $\calC$ is surjective. Then $\calC$ is not a minimal left complement in $\calG$.
\end{theorem}

\begin{proof}
On the contrary, let us assume that $\calC$ is a minimal left complement to a subset $\calS$ of $\calG$. 
Let $\{s_i\}_{i\in \Lambda}$ be elements of $\calS$ such that $\calG = \cup_{i\in \Lambda} \calH s_i$ and 
$$\calH s_i \cap \calH s_j = \emptyset \quad \text{ for all } i \neq j.$$
Since $\calC$ is a proper subset of $\calH$, it follows that for each $i\in \Lambda$, there exists an element $s_i'$ such that $\calH s_i' = \calH s_i$ and $\calC s_i' \neq \calC s_i$. 
For each $(a, i)  \in (\calH \setminus \calC) \times \Lambda$, choose an element $(c_{(a, i)}, s_{(a, i)})$ in $\calC \times \calS$ such that $c_{(a, i)}s_{(a, i)} = as_i$.
Consider the map 
$$(\calH \setminus \calC) \times \Lambda
\to 
\calC$$
defined by 
$$(a, i) \mapsto c_{(a, i)}.$$
Denote the image of this map by $\scrC$.
Note that $\scrC\cdot \calS$ contains $(\calH\setminus \calC) \cdot \{s_i\,|\, i\in \Lambda\}$. Consider the map 
$$\{0, 1\} \times (\calH\setminus \calC) \times \Lambda
\to 
\calC$$
defined by 
$$
(*, a, i) \mapsto 
\begin{cases}
c_{(a, i)} & \text{ if } * = 0, \\
as_i's_i^\mo & \text{ if } * = 1.
\end{cases}
$$
By the hypothesis, $\calC$ contains an element $h$ which lies outside the image of this map, i.e., $h$ avoids the set 
$$\scrC \cup ((\calH\setminus \calC) \cdot  \{s_i's_i^\mo\,|\, i\in \Lambda \}).$$

We claim that 
$$\calH s_i\subseteq 
(\calC \setminus \{h\} ) 
\cdot \calS$$
for any $i\in \Lambda$. Since $\scrC$ is contained in $\calC$, the set $\scrC \cdot \calS$ contains $(\calH\setminus \calC)\cdot s_i$ and $h$ does not lie in $\scrC$, it follows that $(\calH\setminus \calC)\cdot s_i$ is contained in $(\calC \setminus \{h\})\cdot \calS$. Note that $hs_i$ does not lie in $(\calH\setminus \calC) s_i'$. Since $\calH s_i = \calH s_i'$, it follows that $hs_i$ belongs to $\calC \cdot s_i'$. Hence $hs_i$ lies in $(\calC \setminus \{h\}) \cdot s_i'$. Moreover, $(\calC \setminus \{h\}) \cdot s_i$ is contained in $(\calC \setminus \{h\}) \cdot \calS$. This proves the claim that $\calH s_i$ is contained in $(\calC \setminus \{h\}) \cdot \calS$ for any $i\in \Lambda$. Hence $\calC$ is not a minimal left complement to $\calS$. 

If $\calC$ is a minimal right complement to some subset $\calT$ of $\calG$, then $\calC^\mo$ is a minimal left complement to $\calT^\mo$, which is impossible. 
\end{proof}

\begin{remark}
\label{Remark:Cardi}
It follows from Theorem \ref{Thm:Cardi} that given any uncountable group $G$, no proper subset $C$ of $G$ having countable set-theoretic complement in $G$ is a minimal complement in $G$. In particular, 
\begin{enumerate}
\item 
the set of irrational numbers, with a countable number of points removed, is not a minimal complement in $\bbR$, 
\item 
the set of transcendental numbers, with a countable number of points removed, is not a minimal complement in $\bbC$. 
\end{enumerate}
\end{remark}

\section{On robust non-minimal complements} 
\label{Sec:RobustNonMac}

In an abelian group, a minimal complement is often called a minimal additive complement, abbreviated as MAC \cite{BurcroffLuntzlara}. Following \cite[Definition 5]{BurcroffLuntzlara}, one can consider the notion of robust MAC and robust non-MAC in any abelian group $G$. 

\begin{definition}
\label{Defn:Robust}
Let $G$ be an abelian group. A subset $C$ of $G$ is said to be a \textnormal{robust MAC} if any non-empty subset $D$ of $G$ having finite symmetric difference with $C$ is a MAC in $G$. A subset $C$ of $G$ is said to be a \textnormal{robust non-MAC} if any non-empty subset $D$ of $G$ having finite symmetric difference with $C$ is a non-MAC in $G$. 
\end{definition}

Kwon proved that the finite subsets of the integers are robust MACs \cite[Theorem 9]{Kwon}. In \cite{CoMin1}, the authors showed that the finite subsets of any free abelian group of rank $\geq 1$ are robust MACs. Alon--Kravitz--Larson established that the finite subsets in any infinite abelian group are robust MACs \cite[Theorem 2]{AlonKravitzLarson}. Burcroff--Luntzlara proved results which provides several examples of infinite subsets of $\bbZ$ which are robust MACs and several examples of infinite subsets of $\bbZ$ which are robust non-MACs \cite[Theorems 3, 5]{BurcroffLuntzlara}.
As a corollary of Theorem \ref{Thm:QLeavesLAppears}, one also obtains examples of robust non-MACs. 

\begin{corollary}
Suppose $C$ is a subset of $\bbZ$ and it is the union of translates of a nonzero subgroup $K$ of $\bbZ$. 
If $\lambda_\bbZ(C) > 2$, then 
$(C\setminus E) \cup F$ is a robust non-MAC in $\bbZ$ for any finite subset $E$ of $C$ and for any subset $F$ of $\bbZ\setminus C$ such that $(\bbZ\setminus C)\setminus F$ is infinite. 
\end{corollary}
Moreover, there are infinite sets which are neither bounded below nor above, and do not satisfy the hypothesis of \cite[Theorem 5]{BurcroffLuntzlara}. Indeed, consider the set 
$$(\{1, 2, 3\} + 4\bbZ) \cup \calF \cup \calF_1$$
for any subset $\calF$ of $\calF_0$ where the sets $\calF_0, \calF_1$ are defined by 
$$
\calF_i
= 
4\bbZ \setminus 
\{
(4n)! + 4k \,|\, n\geq 1, n\equiv i \,(\mathrm{mod}\, 2), 0 \leq k \leq n
\}
$$
for $i = 0, 1$. Note that $(\{1, 2, 3\} + 4\bbZ) \cup \calF \cup \calF_1$ does not satisfy the hypothesis of \cite[Theorem 5]{BurcroffLuntzlara}. From Theorem \ref{Thm:FContainedInSingleCoset}, it follows that it is a robust non-MAC. Since $\calF_0$ is an infinite set, it has uncountably many subsets. Thus Theorem \ref{Thm:FContainedInSingleCoset} yields examples of uncountably many robust non-MACs in $\bbZ$, none of them satisfying the hypothesis of \cite[Theorem 5]{BurcroffLuntzlara}. This provides a partial answer to \cite[Question 2]{BurcroffLuntzlara}.  More generally, we establish the following Theorem \ref{Thm:RobustNonMac}, which shows in particular that any finitely generated abelian group of positive rank and any free abelian group of positive rank contain uncountably many robust non-MACs. Moreover, it follows from Theorem \ref{Thm:Cardi} that given any uncountable abelian group $G$, any proper subset $C$ of $G$ having countable set-theoretic complement in $G$ is a robust non-MAC.

In the context of groups which are not assumed to be abelian, the minimal complements are more precisely  called minimal multiplicative complements to indicate the underlying structure of the ambient group. The minimal multiplicative complements are abbreviated as MMCs. In the spirit of robust MACs and non-MACs, one can also define the notion of robust MMCs and robust non-MMCs. 

\begin{definition}
\label{Defn:RobustMMC}
Let $G$ be a group. A subset $C$ of $G$ is said to be a \textnormal{robust MMC} if any non-empty subset $D$ of $G$ having finite symmetric difference with $C$ is a MMC in $G$. A subset $C$ of $G$ is said to be a \textnormal{robust non-MMC} if any non-empty subset $D$ of $G$ having finite symmetric difference with $C$ is a non-MMC in $G$. 
\end{definition}

Note that in the context of abelian groups, the MMCs (resp. robust MMCs, robust non-MMCs) coincide with the MACs (resp. robust MACs, robust non-MACs). 

\begin{theorem}
\label{Thm:RobustNonMac}
Any group that admits $\bbZ$ as a quotient, contains uncountably many robust non-MMCs. 
\end{theorem}

\begin{proof}
Note that there exists a normal subgroup $G'$ of $G$ such that $G/G'$ is isomorphic to $\bbZ$. Let $p\geq 5$ be a prime number. Let $a$ be a positive integer satisfying $p \geq 3a + 1$ and $\scrC$ denote a subset of $\{1, 2, \cdots, p\}$ of size $p-a$. Denote the set $\{1, 2, \cdots, p\}\setminus \scrC$ by $\scrC'$. Let $\psi: G/G' \to \bbZ$ be a group isomorphism. 
Let $K$ denote the subgroup $\psi^\mo (p\bbZ)$ of $G$. Note that $K$ is a normal subgroup of $G$. For any subset $\calF$ of $\scrC' + p\bbN$, the subset $\psi^\mo ((\scrC + p\bbZ )\cup \calF)$ is a robust non-MAC by Proposition \ref{Prop:CMinusCSymmPrimeIndex} and Theorem \ref{Thm:CMinusCSymm}. Since the set $\scrC' + p\bbN$ contains infinitely many elements, it has uncountably many subsets. Thus the group $G$ contains uncountably many robust non-MACs. 
\end{proof}

\begin{corollary}
For any number field $K$ of degree $\geq 3$, the group $\gln_n(\mathcal O_K)$ contains uncountably many robust non-MMCs where $\calO_K$ denotes the ring of integers of $K$. 
\end{corollary}

\begin{proof}
Note that the group $\gln_n(\calO_K)$ admits $\calO_K^\times$ as a quotient. Since $K$ has degree $\geq 3$, by the Dirichlet's unit theorem, $\calO_K^\times$ admits $\bbZ$ as a quotient. Hence $\gln_n(\calO_K)$ admits $\bbZ$ as a quotient. By Theorem \ref{Thm:RobustNonMac}, the result follows. 
\end{proof}

\begin{corollary}
Any finitely generated abelian group of positive rank and any free abelian group of positive rank contain uncountably many robust non-MACs.
\end{corollary}

\begin{proof}
It follows from Theorem \ref{Thm:RobustNonMac}.
\end{proof}

It seems plausible that any infinite abelian group contains uncountably many robust non-MACs. 

We conclude this section with the following remarks. 

\begin{remark}
If a subset $A$ of $H$ (for example, the sets of form $C\cup F$ considered in Section \ref{Sec:NonMinComp}) is not a minimal left complement in $G$, then so are its left translates, i.e., the sets of the form $g\cdot A$ for any $g\in G$. 
\end{remark}

\begin{remark}
Note that the subsets which are shown to be non-minimal complements are not a part of any co-minimal pair\footnote{A pair $(A, B)$ of nonempty subsets of a group $G$ is called a \textit{co-minimal pair} if $A$ is a minimal left complement to $B$ and $B$ is a minimal right complement to $A$ \cite[Definition 1.1]{CoMin1}.}. By Theorem \ref{Thm:RobustNonMac}, any finitely generated abelian group of positive rank contains uncountably many infinite subsets which are robust non-MACs and in particular, not minimal complements. We contrast this result with \cite[Theorem 2.2]{CoMin3}, which states that any such group also contains uncountably many infinite subsets which admit minimal complements. 
In \cite{CoMin3}, we considered lacunary sequences\footnote{A sequence $t_0 < t_1 < t_2 < \cdots $ of elements of $\bbZ^d$ is said to be \textit{lacunary} if 
$t_0 > 0$ and for some positive integer $\lambda \geq 2$, $t_n > \lambda t_{n-1}$ for any $n\geq 1$, where ``$>$'' denotes the lexicographic order on $\bbZ^d$.
} in $\bbZ^d$ for $d\geq 1$, and proved that ``a majority'' of such sequences are a part of a co-minimal pair, and in particular, they are minimal complements \cite[Theorem 2.1]{CoMin3}. It also follows that any such sequence remain a minimal complement even after the removal of finitely many points. It would be interesting to investigate whether ``a majority'' of such sequences are robust MACs. It follows from \cite[Theorem 4]{BurcroffLuntzlara} that it is indeed the case when $d = 1$.

\end{remark}

It would be interesting to consider the situations when the sets $C, F$ (as in Section \ref{Sec:NonMinComp}) are somewhat modified. For instance, 
\begin{enumerate}
\item 
What happens if $C$ is taken to be a ``large'' set in $H$ and $F$ is taken to be a ``small'' set lying outside $H$ (or more generally, to be a ``small'' set in $G\setminus C$)? 
\item What happens if $C$ intersects with several right cosets of $H$ in $G$ and the intersection of $C$ with each such (or one such) right coset is ``large''?
\end{enumerate}

\section{Acknowledgements}
The first author is supported by the ISF Grant no. 662/15. He wishes to thank the Department of Mathematics at the Technion where a part of the work was carried out. The second author would like to acknowledge the Initiation Grant from the Indian Institute of Science Education and Research Bhopal, and the INSPIRE Faculty Award from the Department of Science and Technology, Government of India.

\def\cprime{$'$} \def\Dbar{\leavevmode\lower.6ex\hbox to 0pt{\hskip-.23ex
  \accent"16\hss}D} \def\cfac#1{\ifmmode\setbox7\hbox{$\accent"5E#1$}\else
  \setbox7\hbox{\accent"5E#1}\penalty 10000\relax\fi\raise 1\ht7
  \hbox{\lower1.15ex\hbox to 1\wd7{\hss\accent"13\hss}}\penalty 10000
  \hskip-1\wd7\penalty 10000\box7}
  \def\cftil#1{\ifmmode\setbox7\hbox{$\accent"5E#1$}\else
  \setbox7\hbox{\accent"5E#1}\penalty 10000\relax\fi\raise 1\ht7
  \hbox{\lower1.15ex\hbox to 1\wd7{\hss\accent"7E\hss}}\penalty 10000
  \hskip-1\wd7\penalty 10000\box7}
  \def\polhk#1{\setbox0=\hbox{#1}{\ooalign{\hidewidth
  \lower1.5ex\hbox{`}\hidewidth\crcr\unhbox0}}}
\providecommand{\bysame}{\leavevmode\hbox to3em{\hrulefill}\thinspace}
\providecommand{\MR}{\relax\ifhmode\unskip\space\fi MR }
\providecommand{\MRhref}[2]{%
  \href{http://www.ams.org/mathscinet-getitem?mr=#1}{#2}
}
\providecommand{\href}[2]{#2}

\end{document}